\documentclass[12pt]{amsart}
\usepackage{amscd, amsfonts, amssymb, amsthm}
\usepackage{parskip, fullpage, verbatim}
\usepackage[colorlinks=true]{hyperref}
\usepackage[all]{xypic}

\newcommand\CA{{\mathcal A}}

\newcommand\BBC{{\mathbb C}}

\newcommand\BBZ{{\mathbb Z}}

\newcommand {\GAP}{\textsf{GAP}}  
\newcommand {\CHEVIE}{\textsf{CHEVIE}}  
\newcommand {\Singular}{\textsf{SINGULAR}}

\newcommand\Der{{\operatorname{Der}}}

\newcommand\GL{\operatorname{GL}}

\newcommand\pdeg{\operatorname{pdeg}}

\numberwithin{equation}{section}

\theoremstyle{plain}
\newtheorem{lemma}[equation]{Lemma}
\newtheorem{theorem}[equation]{Theorem}


\begin{document}

\title[Reflection arrangements are hereditarily free]
{Reflection arrangements are hereditarily free}

\author[T. Hoge]{Torsten Hoge}
\address
{Fakult\"at f\"ur Mathematik,
Ruhr-Universit\"at Bochum,
D-44780 Bochum, Germany}
\email{torsten.hoge@rub.de}

\author[G. R\"ohrle]{Gerhard R\"ohrle}
\address
{Fakult\"at f\"ur Mathematik,
Ruhr-Universit\"at Bochum,
D-44780 Bochum, Germany}
\email{gerhard.roehrle@rub.de}


\allowdisplaybreaks

\begin{abstract}
Suppose that $W$ is a finite, unitary, 
reflection group acting on the complex 
vector space $V$.
Let $\CA = \CA(W)$ be the associated 
hyperplane arrangement of $W$.
Terao has shown that 
each such reflection
arrangement $\CA$ is free.
Let $L(\CA)$ be the intersection lattice of $\CA$.
For a subspace $X$ in $L(\CA)$ we have the restricted
arrangement $\CA^X$ in $X$ by means of restricting hyperplanes from $\CA$
to $X$.
In 1992, Orlik and Terao 
conjectured that each such restriction is again free.
In this note
we settle the outstanding cases 
confirming the conjecture.

In 1992, 
Orlik and Terao also conjectured that every reflection arrangement is 
hereditarily inductively free.
In contrast, this stronger conjecture is false however; 
we give two counterexamples.
\end{abstract}

\maketitle


\section{Introduction}

Suppose that $W$ is a finite, unitary, 
reflection group acting on the complex 
vector space $V$.
Let $\CA = (\CA,V) = (\CA(W),V)$ be the associated 
hyperplane arrangement of $W$.
Terao \cite{terao:freeI} has shown that each reflection
arrangement $\CA$ is free and that 
the multiset of exponents 
$\exp \CA$ of $\CA$ is given by the 
coexponents of $W$
(see also \cite[\S 6]{orlikterao:arrangements}).

Let $L(\CA)$ be the intersection lattice of $\CA$.
For a subspace $X \in L(\CA)$ we consider the restricted
arrangement $\CA^X = (\CA^X, X)$ in $X$ by means of restricting hyperplanes from $\CA$
to $X$.
In 1992, Orlik and Terao \cite[Conj.\ 6.90]{orlikterao:arrangements}
conjectured that each such restriction is again a free arrangement.
Free arrangement with this property are called \emph{hereditarily free},
 \cite[Def.\ 4.140]{orlikterao:arrangements}. 
In this note we settle the outstanding cases 
confirming the conjecture:

\begin{theorem}
\label{main}
For $W$ a finite complex reflection group,  
the reflection arrangement $\CA = \CA(W)$ is 
hereditarily free.
\end{theorem}

Note that in general 
the restriction of a free arrangement need not be
free again (cf.~\cite[Ex.\ 4.141]{orlikterao:arrangements}).

Thanks to \cite[Prop.\ 4.28]{orlikterao:arrangements}, 
Theorem \ref{main}  reduces 
readily to the case when $\CA = \CA(W)$ is irreducible.

In  \cite[Prop.\ 6.73, Prop.\ 6.77, Cor.\ 6.86]{orlikterao:arrangements},
Orlik and Terao proved that each 
restricted arrangement $\CA^X$ is again free
provided that $W$ is a symmetric group or a monomial group $G(r,p,\ell)$.
The case when $W$ is a cyclic group is trivial.
This settles Theorem \ref{main} 
for each of the infinite series of complex reflection groups.

Furthermore, 
in a case-by-case study, 
Orlik and Terao showed in \cite{orlikterao:free}
that 
$\CA^X$ is free when $W$ is a Coxeter group.
In case $W$ is a Weyl group, Douglass \cite{douglass:adjoint}
gave a uniform proof
of this fact using an elegant conceptual Lie theoretic argument.

Moreover, for any hyperplane arrangement $\CA$, it is known that
$\CA^X$ is free 
in case $\dim X = 1$, 
\cite[Def.\ 4.7; Prop.\ 4.27]{orlikterao:arrangements}, 
as well as when
$\dim X = 2$ \cite[Ex.\ 4.20]{orlikterao:arrangements}.

For the exceptional complex reflection groups, 
Orlik and Terao checked that in each 
instance when $\dim X = 3$,  
the restriction $\CA^X$ is again free  
\cite[App.\ D]{orlikterao:arrangements}.

In this note, we settle the remaining instances in the 
exceptional groups.
There are only four instances when the freeness of $\CA^X$ still 
needs to be checked: Either $W = G_{33}$ and $X$ is a hyperplane in $V$
(there is only one class of hyperplanes and here $\dim X = 4$), or else
$W = G_{34}$ and $X$ is a hyperplane in $V$ 
(there is only one class of hyperplanes and here $\dim X = 5$), 
or $X \in L(\CA)$ belongs to one of two classes of $4$-dimensional 
subspaces in $V$.

Our proof of these remaining cases for Theorem \ref{main} is computational.
First we use the functionality for complex reflection groups 
provided by the   \CHEVIE\ package in   \GAP\ 
(and some \GAP\ code by J.~Michel)
(see \cite{gap3} and \cite{chevie})
in order to obtain explicit 
linear functionals $\alpha$ defining the hyperplanes 
$\ker \alpha$ of the underlying reflection arrangement
$\CA(W)$. 
These then allow us to implement the 
module of derivations $D(\alpha)$ associated with $\alpha$
in the   \Singular\ computer algebra system (cf.~\cite{singular}). 
Then the
module theoretic functionality of
  \Singular\ is used to show that the
modules of derivations in question $D(\CA^X)$ 
are free.

While our calculations 
(combined with the existing known instances of the 
conjecture of Orlik and Terao)  
do provide a proof of Theorem \ref{main}, it would  
nevertheless be very desirable to have a uniform,  conceptual
proof free of case-by-case considerations and free of 
computer calculations.
A  conceptual proof is only known  in case  
of Weyl groups \cite{douglass:adjoint}.

The notion of freeness was introduced by Saito in his 
seminal work \cite{saito}.
Questions of freeness play a central role in 
the understanding of arrangements (see 
\cite[\S 4]{orlikterao:arrangements}, \cite{terao:freeI}).
In current research they are still of key importance;
for instance in form of 
inductively free arrangements,
e.g., see \cite{cuntz:indfree} or  \cite{hogeroehrle:indfree},
or in the context of multiarrangements, 
e.g., see \cite{schulz:free}.

In the next section, we recall the required notation and  
facts about freeness of hyperplane arrangements and 
reflection arrangements from
\cite[\S 4, \S6]{orlikterao:arrangements}.
We settle the outstanding cases of Orlik and Terao's conjecture,
completing the proof of Theorem \ref{main} in Section 3.

For general information about arrangements and reflection groups, we refer
the reader to \cite{orlikterao:arrangements} and \cite{bourbaki:groupes}.

\section{Preliminaries}

\subsection{Hyperplane Arrangements}
\label{ssect:hyper}

Let $V = \BBC^\ell$ 
be an $\ell$-dimensional complex vector space. 
A \emph{hyperplane arrangement} is a pair
$(\CA, V)$, where $\CA$ is a finite collection of hyperplanes in $V$.
Frequently, we simply write $\CA$ in place of $(\CA, V)$.
The \emph{lattice of $\CA$}, $L(\CA)$, is the set of subspaces of $V$ of
the form $H_1\cap \dotsm \cap H_n$, where $\{ H_1, \ldots, H_n \}$ is a subset
of $\CA$. 

For $X \in L(\CA)$, we have two associated arrangements, 
firstly the subarrangement 
$\CA_X :=\{H \in \CA \mid X \subseteq H\}$
of $\CA$ and secondly, 
the \emph{restriction of $\CA$ to $X$}, $(\CA^X,X)$, where 
$\CA^X := \{ X \cap H \mid H \in \CA \setminus \CA_X\}$.

Let $S = S(V^*)$ be the symmetric algebra of the dual space $V^*$ of $V$.
If $x_1, \ldots , x_\ell$ is a basis of $V^*$, then we identify $S$ with 
the polynomial ring $\BBC[x_1, \ldots , x_\ell]$.
Letting $S_p$ denote the $\BBC$-subspace of $S$
consisting of the homogeneous polynomials of degree $p$ (along with $0$),
we see that
$S$ is naturally $\BBZ$-graded: $S = \bigoplus_{p \in \BBZ}S_p$, where
$S_p = 0$ in case $p < 0$.

Let $\Der(S)$ be the $S$-module of $\BBC$-derivations of $S$, where the
$S$-module structure is defined as follows:
For $f \in S$ and $\theta_1, \theta_2 \in \Der(S)$, let 
$f\theta_1, \theta_1 + \theta_2 \in \Der(S)$ be defined by 
$(f\theta_1)(g) := f \cdot \theta_1(g)$ and
$(\theta_1 + \theta_2)(g) = \theta_1(g) + \theta_2(g)$ for $g \in S$.
For $i = 1, \ldots, \ell$, 
let $D_i := \partial/\partial x_i$ be the usual derivation of $S$.
Then $D_1, \ldots, D_\ell$ is a $\BBC$-basis of $\Der(S)$.
We say that $\theta \in \Der(S)$ is 
\emph{homogeneous of polynomial degree p}
provided 
$\theta = \sum_{i=1}^\ell f_i D_i$, 
where $f_i \in S_p$ for each $1 \le i \le \ell$.
In this case we write $\pdeg \theta = p$.
Let $\Der(S)_p$ be the $\BBC$-subspace of $\Der(S)$ consisting 
of all homogeneous derivations of polynomial degree $p$.
Then $\Der(S)$ is a graded $S$-module:
$\Der(S) = \bigoplus_{p\in \BBZ} \Der(S)_p$.
For instance, the \emph{Euler derivation}
$\theta_E := \sum_{i=1}^\ell x_i D_i$ is homogeneous of 
polynomial degree $1$
(cf.~\cite[Def.~4.7]{orlikterao:arrangements}).

Following \cite[Def.~4.4]{orlikterao:arrangements}, 
for $f \in S$, we define the $S$-submodule $D(f)$ of $\Der(S)$ by
\[
D(f) := \{\theta \in \Der(S) \mid \theta(f) \in f S\} .
\]

Let $\CA$ be an arrangement in $V$. 
Then for $H \in \CA$ we fix $\alpha_H \in V^*$ with
$H = \ker(\alpha_H)$.
The \emph{defining polynomial} $Q(\CA)$ of $\CA$ is given by 
$Q(\CA) := \prod_{H \in \CA} \alpha_H \in S$.

The \emph{module of $\CA$-derivations} of $\CA$ is 
defined by 
\[
D(\CA) := D(Q(\CA)).
\]
Note that for any arrangement $\CA$ we have that $\theta_E \in D(\CA)$
(cf.~\cite[Def.~4.7]{orlikterao:arrangements}).
We say that $\CA$ is \emph{free} if the module of $\CA$-derivations
$D(\CA)$ is a free $S$-module.

With the $\BBZ$-grading of $\Der(S)$, the module of $\CA$-derivations
becomes a graded $S$-module $D(\CA) = \bigoplus_{p\in \BBZ} D(\CA)_p$,
where $D(\CA)_p = D(\CA) \cap \Der(S)_p$ 
\cite[Prop.\ 4.10]{orlikterao:arrangements}.
If $\CA$ is a free arrangement, then the $S$-module 
$D(\CA)$ admits a basis of $\ell$ homogeneous derivations, 
say $\theta_1, \ldots, \theta_\ell$ 
\cite[Prop.\ 4.18]{orlikterao:arrangements}.
While the $\theta_i$'s are not unique, their polynomial 
degrees $\pdeg \theta_i$ 
are unique (up to ordering). This multiset is the set of 
\emph{exponents} of the free arrangement $\CA$
and is denoted by $\exp \CA$.

An important theorem for free arrangements $\CA$  states that 
the Poincar\'e polynomial $\pi(\CA,t)$ of the lattice $L(\CA)$ 
(cf.\ \cite[\S 2.3]{orlikterao:arrangements})
factors into linear terms given by the exponents of $\CA$ as
\[
\pi(\CA,t) = \prod_{i=1}^\ell (1 + b_i t),
\]
where $\exp \CA = \{b_1, \ldots, b_\ell\}$ are the exponents of $\CA$
\cite[Thm.\ 4.137]{orlikterao:arrangements}.
This factorization property suggests
that freeness of $\CA$ only depends on the 
lattice $L(\CA)$; this is a basic conjecture due to Terao 
\cite[Conj.\ 4.138]{orlikterao:arrangements}.

Following \cite[Def.\ 4.140]{orlikterao:arrangements}, 
we say that $\CA$ is \emph{hereditarily free}
provided $\CA^X$ is free for every $X \in L(\CA)$.
In general, a free arrangement need not be hereditarily free,
thanks to a  counterexample due to Edelman and Reiner  
\cite[Ex.\ 4.141]{orlikterao:arrangements}.

\subsection{Reflection Arrangements}
\label{ssect:refl}

Suppose that $W \subseteq \GL(V)$ 
is a finite, complex reflection group acting on the complex
vector space $V=\BBC^\ell$.
The \emph{reflection arrangement} $\CA = \CA(W)$ of $W$ in $V$ is 
the hyperplane arrangement 
consisting of the reflecting hyperplanes of the elements in $W$
acting as reflections on $V$.

It is known that each reflection arrangement 
$\CA(W)$ is free \cite{terao:freeI}.
In \cite[Prop.\ 6.89]{orlikterao:arrangements}, 
Orlik and Terao proved in a case-by-case argument
that, for any $X \in L(\CA)$ with $\dim X = p$,
there  exist integers $b_1^X, \ldots, b_p^X$ 
such that the Poincar\'e polynomial 
of the restriction $\CA^X$ satisfies the factorization property
\[
\pi(\CA^X,t) = \prod_{i=1}^p (1 + b_i^X t).
\]
Moreover, in all instances when $\CA^X$ is known to be free, 
the equality 
$\exp \CA^X = \{b_1^X, \ldots, b_p^X\}$ holds.
In view of this fact and the aforementioned factorization 
theorem for free arrangements, 
\cite[Thm.\ 4.137]{orlikterao:arrangements}, 
Orlik and Terao conjecture that every 
reflection arrangement is hereditarily free
(cf.\ \cite[Conj.\ 6.90]{orlikterao:arrangements}).
Theorem \ref{main} settles this conjecture.

\section{Proof of Theorem \ref{main}}

As explained in the Introduction, 
all but four cases of Theorem \ref{main} have already been proved.
We thus concentrate on the four outstanding 
incidences in $G_{33}$ and $G_{34}$.

In the case of the restriction of $\CA$ to a hyperplane $H$,
one is inclined to use the Addition-Deletion Theorem
\cite[Thm.\ 4.51]{orlikterao:arrangements} 
in order to derive the freeness of $\CA^H$.
However, in the case of $G_{33}$
and $G_{34}$, 
the set $\{b_1^H, \ldots, b_p^H\}$
is not a subset of $\exp \CA(W)$ 
(cf.~\cite[Tables C.14, C.17]{orlikterao:arrangements}), 
and thus this criterion does not apply.

Recall that the defining polynomial 
$Q(\CA) = \prod_{H \in \CA} \alpha_H$ of $\CA$ is a product of linear factors.
Thanks to \cite[Prop.~4.8]{orlikterao:arrangements}, we have
\begin{equation}
\label{eq}
  D(\CA) = \bigcap_{H \in \CA} D(\alpha_H).
\end{equation} 
Therefore, we may compute $D(\CA)$ 
as an intersection of the $S$-modules $D(\alpha_H)$.

The linear factors can be obtained from the 
  \CHEVIE\  package in   \GAP\  (see \cite{gap3} and \cite{chevie})
in the following way:
\begin{itemize}
\item \verb|G:=ComplexReflectionGroup(33);| \\
returns the complex reflection group with Shephard-Todd number $33$:  $G = G_{33}$.
\item \verb|R:=Reflections(G);| \\
returns the list $R$ of reflections of $G$ (some of which occur more than once).
\item \verb|A:=MatXPerm(G,g);| \\
returns the representation matrix of $g \in R$ on $V$.
\item \verb|BaseMat(A-IdentityMat(l))[1];| \\
returns the linear form $\alpha$ whose kernel 
      is the corresponding hyperplane.
\end{itemize}

Let $e_1, \ldots, e_\ell \in V$ be the dual basis of 
$x_1, \ldots, x_\ell \in S$.
We define a map ${}^\vee : V \to \Der(S)_0$ by
$v = \sum_{i = 1}^\ell \lambda_i e_i \mapsto 
v^\vee := \sum_{i = 1}^\ell \lambda_i D_i \in \Der(S)_0$.

Using this notation, 
a set of generators of $D(\alpha)$ is given by a 
basis of the corresponding hyperplane along with the 
Euler derivation $\theta_E$ as follows:

\begin{lemma}
\label{lem}
  Let $\{v_1,\ldots,v_{\ell-1}\}$ be a $\BBC$-basis 
of $\ker \alpha$ for $\alpha \in  V^*\setminus\{0\}$.
Then 
$\{\theta_E,v_1^\vee,\ldots,v_{\ell-1}^\vee \}$ is an $S$-basis of $D(\alpha)$.

In particular, we have $D(\CA)_0 = \bigcap_{H \in \CA} H$.
\end{lemma}

\begin{proof}
Note that 
we have $v_i^\vee(\alpha) = \alpha(v_i) = 0$, and so $v_i^\vee \in D(\alpha)$, 
for each $i = 1, \ldots , \ell-1$.
Let $M := M(\theta_E, v_1^\vee,\ldots,v_{\ell-1}^\vee)$ be the 
coefficient matrix associated with 
$\theta_E, v_1^\vee,\ldots,v_{\ell-1}^\vee \in D(\alpha)$,
i.e., the entries of $M$ are the coefficients of 
$\theta_E,v_1^\vee,\ldots,v_{\ell-1}^\vee$ in terms of the $D_i$'s 
(cf.\ \cite[Def.~4.11]{orlikterao:arrangements}).
Thanks to \cite[Prop.~4.12]{orlikterao:arrangements},
we have $\det M \in \alpha S$.
Therefore, once we know that $\det M\not= 0$,
we get that $\deg (\det M) = 1$, and thus 
$\det M = \lambda \alpha$ for some $\lambda \in \BBC\setminus\{0\}$.

Computing $\det M$ by Laplace along the column given by $\theta_E$, it
is obvious that the determinant does not vanish, since 
$\{v_1,\ldots,v_{\ell-1}\}$ is linearly independent and therefore at least one 
$(\ell-1)$-minor of the coefficient-matrix of the $v_i^\vee$'s is non-zero.
Consequently, $\det M = \lambda \alpha$ for some 
$\lambda \in \BBC\setminus\{0\}$.
The result now follows from Saito's criterion  
\cite[Thm.~4.19]{orlikterao:arrangements}.
\end{proof}

In order to compute $D(\CA^H)$, we 
require the defining polynomial $Q(\CA^H)$ along with its linear
factors. The strategy is to insert the equation 
defining the hyperplane $H$ 
into the remaining factors of $Q(\CA)$. Some of the resulting new 
factors then coincide modulo a scalar, and consequently, 
the corresponding new hyperplanes are 
the same. For each hyperplane, we choose only one such factor for the defining 
polynomial of $\CA^H$. 

Using the explicit data provided by \CHEVIE\  and Lemma \ref{lem}, 
we can calculate $D(\CA^H)$ using equation \eqref{eq}
along with 
the \verb|intersect| command in   \Singular\ (cf.\ \cite{singular}).

The derivations that form a basis of $D(\CA^X)$ that we have calculated 
using \Singular\ in the four cases above
are  expressions with long and complicated polynomial coefficients, 
with the exception of the Euler derivation, of course. 
So they are simply too cumbersome 
and not particularly enlightening in order to be listed explicitly. 
The interested reader may find them  
via the link   
\url{http://www.ruhr-uni-bochum.de/ffm/Lehrstuehle/Lehrstuhl-VI/hyperplane_arrangements.html}.

The following algorithmic method to show that a given arrangement
$\CA$ is free is proposed by Barakat and Cuntz 
\cite[\S 6.3]{cuntz:indfree}.
Start with the empty arrangement and successively add hyperplanes. 
At each step check if the module of derivations given by the corresponding 
intersections is free.
This algorithm only works if $\CA$ is \emph{inductively free}
\cite[Def.\ 4.53]{orlikterao:arrangements}. In that case there is 
an ordering of the hyperplanes 
$H_1, H_2, \ldots$
such that, 
for each subarrangement $\CA_i:=\{H_1, \ldots, H_i\}$ of $\CA$, 
the corresponding module
of derivations $D(\CA_i)$ is free. 
Unfortunately, neither $\CA(G_{33})$ nor $\CA(G_{34})$ is  
inductively free; see the next paragraph. 
So it is not possible to employ the algorithm from 
\cite{cuntz:indfree} in our case.

We now show that $\CA = \CA(G_{33})$ is not inductively free;
for $G_{34}$ the argument is similar.
Our argument depends on the result of our computation that 
the restricted arrangement $\CA^H$ is free for $H \in \CA$. 
For, since $\CA$ and $\CA^H$ are free, it follows from 
\cite[Prop.~4.57]{orlikterao:arrangements} that the map 
$q\colon D(\CA) \rightarrow D(\CA^H)$ 
(as defined in \cite[Prop.\ 4.45]{orlikterao:arrangements}) 
is surjective if and only if 
the subarrangement $\CA\setminus\{H\}$  of $\CA$ is free. 
Since the map $q$ is ``degree preserving'' 
(\cite[Prop.\ 4.44]{orlikterao:arrangements}), 
it follows from the list of 
the polynomial
degrees of the generators of the free modules
$D(\CA)$ and $D(\CA^H)$ (i.e., the exponents
of $\CA$ and $\CA^H$) given 
in \cite[Table C.14]{orlikterao:arrangements} that $q$ is not onto. 
This in particular shows that $\CA\setminus\{H\}$ is not free and so 
$\CA$ is not inductively free.
This in particular  shows  that Orlik and Terao's conjecture
that every reflection arrangement is inductively free
is false,  \cite[Conj.\ 6.91]{orlikterao:arrangements}.
However, only recently, Barakat and Cuntz showed in \cite{cuntz:indfree}
that every Coxeter arrangement is 
inductively free.
In the forthcoming paper \cite{hogeroehrle:indfree},
we classify all inductively free reflection arrangements.

As indicated above, 
the computations to 
settle the four open cases in order to complete the proof of 
Theorem \ref{main} were done using  \Singular.
They were carried out on a 4 x Intel Quad Core Xeon E7340 / 2,4 GHz with 
128 GB RAM. The three cases where one has to restrict to a subspace of
dimension 4 were computed in less than 2 minutes each. 
The most elaborate case was to 
calculate
$D(\CA(G_{34})^H)$. This was computed in 2 days. 
However, the calculation here might 
take longer,  since   \Singular\  uses some random choices in its 
use of Gr\"obner bases constructions.


\medskip 
{\bf Acknowledgments}: 
We 
acknowledge 
support from the DFG-priority program 
SPP1489 ``Algorithmic and Experimental Methods in
Algebra, Geometry, and Number Theory''.


\bigskip

\bibliographystyle{amsalpha}

\newcommand{\etalchar}[1]{$^{#1}$}
\providecommand{\bysame}{\leavevmode\hbox to3em{\hrulefill}\thinspace}
\providecommand{\MR}{\relax\ifhmode\unskip\space\fi MR }
\providecommand{\MRhref}[2]{%
  \href{http://www.ams.org/mathscinet-getitem?mr=#1}{#2} }
\providecommand{\href}[2]{#2}


\end{document}